\documentclass[12pt]{amsart}
\usepackage{amsmath,amssymb,amsbsy,amsfonts,latexsym,amsopn,amstext,cite,
                                               amsxtra,euscript,amscd,bm,mathabx}
\usepackage{url}
\usepackage[colorlinks,linkcolor=blue,anchorcolor=blue,citecolor=blue,backref=page]{hyperref}
\usepackage{color}
\usepackage{graphics,epsfig}
\usepackage{graphicx}
\usepackage{float}
\usepackage[english]{babel}
\usepackage{mathtools}
\usepackage{todonotes}
\usepackage{url}
\usepackage[colorlinks,linkcolor=blue,anchorcolor=blue,citecolor=blue,backref=page]{hyperref}

\usepackage[norefs,nocites]{refcheck}

\hypersetup{breaklinks=true}

\usepackage[norefs,nocites]{refcheck}
\usepackage[english]{babel}
\begin{document}

\newtheorem{thm}{Theorem}
\newtheorem{lem}[thm]{Lemma}
\newtheorem{claim}[thm]{Claim}
\newtheorem{cor}[thm]{Corollary}
\newtheorem{prop}[thm]{Proposition}
\newtheorem{definition}[thm]{Definition}
\newtheorem{rem}[thm]{Remark}
\newtheorem{question}[thm]{Open Question}
\newtheorem{conj}[thm]{Conjecture}
\newtheorem{prob}{Problem}

\newtheorem{lemma}[thm]{Lemma}

\newcommand{\GL}{\operatorname{GL}}
\newcommand{\SL}{\operatorname{SL}}
\newcommand{\lcm}{\operatorname{lcm}}
\newcommand{\ord}{\operatorname{ord}}
\newcommand{\Op}{\operatorname{Op}}
\newcommand{\Tr}{\operatorname{Tr}}
\newcommand{\Nm}{\operatorname{Nm}}

\numberwithin{equation}{section}
\numberwithin{thm}{section}
\numberwithin{table}{section}

\def\vol {{\mathrm{vol\,}}}
\def\squareforqed{\hbox{\rlap{$\sqcap$}$\sqcup$}}
\def\qed{\ifmmode\squareforqed\else{\unskip\nobreak\hfil
\penalty50\hskip1em\null\nobreak\hfil\squareforqed
\parfillskip=0pt\finalhyphendemerits=0\endgraf}\fi}

\def \balpha{\bm{\alpha}}
\def \bbeta{\bm{\beta}}
\def \bgamma{\bm{\gamma}}
\def \blambda{\bm{\lambda}}
\def \bchi{\bm{\chi}}
\def \bphi{\bm{\varphi}}
\def \bpsi{\bm{\psi}}
\def \bomega{\bm{\omega}}
\def \btheta{\bm{\vartheta}}

\newcommand{\bfxi}{{\boldsymbol{\xi}}}
\newcommand{\bfrho}{{\boldsymbol{\rho}}}

\def\Kab{\sfK_\psi(a,b)}
\def\Kuv{\sfK_\psi(u,v)}
\def\SaUV{\cS_\psi(\balpha;\cU,\cV)}
\def\SaAV{\cS_\psi(\balpha;\cA,\cV)}

\def\SUV{\cS_\psi(\cU,\cV)}
\def\SAB{\cS_\psi(\cA,\cB)}

\def\Kmnp{\sfK_p(m,n)}

\def\KKap{\cH_p(a)}
\def\KKaq{\cH_q(a)}
\def\KKmnp{\cH_p(m,n)}
\def\KKmnq{\cH_q(m,n)}

\def\Klmnp{\sfK_p(\ell, m,n)}
\def\Klmnq{\sfK_q(\ell, m,n)}

\def \SALMNq {\cS_q(\balpha;\cL,\cI,\cJ)}
\def \SALMNp {\cS_p(\balpha;\cL,\cI,\cJ)}

\def \SACXMQX {\fS(\balpha,\bzeta, \bxi; M,Q,X)}

\def\SAMJp{\cS_p(\balpha;\cM,\cJ)}
\def\SAMJq{\cS_q(\balpha;\cM,\cJ)}
\def\SAqMJq{\cS_q(\balpha_q;\cM,\cJ)}
\def\SAJq{\cS_q(\balpha;\cJ)}
\def\SAqJq{\cS_q(\balpha_q;\cJ)}
\def\SAIJp{\cS_p(\balpha;\cI,\cJ)}
\def\SAIJq{\cS_q(\balpha;\cI,\cJ)}

\def\RIJp{\cR_p(\cI,\cJ)}
\def\RIJq{\cR_q(\cI,\cJ)}

\def\TWXJp{\cT_p(\bomega;\cX,\cJ)}
\def\TWXJq{\cT_q(\bomega;\cX,\cJ)}
\def\TWpXJp{\cT_p(\bomega_p;\cX,\cJ)}
\def\TWqXJq{\cT_q(\bomega_q;\cX,\cJ)}
\def\TWJq{\cT_q(\bomega;\cJ)}
\def\TWqJq{\cT_q(\bomega_q;\cJ)}

 \def \xbar{\overline x}
  \def \ybar{\overline y}

\def\cA{{\mathcal A}}
\def\cB{{\mathcal B}}
\def\cC{{\mathcal C}}
\def\cD{{\mathcal D}}
\def\cE{{\mathcal E}}
\def\cF{{\mathcal F}}
\def\cG{{\mathcal G}}
\def\cH{{\mathcal H}}
\def\cI{{\mathcal I}}
\def\cJ{{\mathcal J}}
\def\cK{{\mathcal K}}
\def\cL{{\mathcal L}}
\def\cM{{\mathcal M}}
\def\cN{{\mathcal N}}
\def\cO{{\mathcal O}}
\def\cP{{\mathcal P}}
\def\cQ{{\mathcal Q}}
\def\cR{{\mathcal R}}
\def\cS{{\mathcal S}}
\def\cT{{\mathcal T}}
\def\cU{{\mathcal U}}
\def\cV{{\mathcal V}}
\def\cW{{\mathcal W}}
\def\cX{{\mathcal X}}
\def\cY{{\mathcal Y}}
\def\cZ{{\mathcal Z}}
\def\Ker{{\mathrm{Ker}}}

\def\NmQR{N(m;Q,R)}
\def\VmQR{\cV(m;Q,R)}

\def\Xm{\cX_m}

\def \A {{\mathbb A}}
\def \B {{\mathbb A}}
\def \C {{\mathbb C}}
\def \F {{\mathbb F}}
\def \G {{\mathbb G}}
\def \L {{\mathbb L}}
\def \K {{\mathbb K}}
\def \N {{\mathbb N}}
\def \PP {{\mathbb P}}
\def \Q {{\mathbb Q}}
\def \R {{\mathbb R}}
\def \Z {{\mathbb Z}}
\def \fS{\mathfrak S}

\def \fP{\mathfrak P}

\def\e{{\mathbf{\,e}}}
\def\ep{{\mathbf{\,e}}_p}
\def\eq{{\mathbf{\,e}}_q}

\def\\{\cr}
\def\({\left(}
\def\){\right)}
\def\fl#1{\left\lfloor#1\right\rfloor}
\def\rf#1{\left\lceil#1\right\rceil}

\def\Tr{{\mathrm{Tr}}}
\def\Nm{{\mathrm{Nm}}}
\def\Im{{\mathrm{Im}}}

\def \oF {\overline \F}

\newcommand{\pfrac}[2]{{\left(\frac{#1}{#2}\right)}}

\def \Prob{{\mathrm {}}}
\def\e{\mathbf{e}}
\def\ep{{\mathbf{\,e}}_p}
\def\epp{{\mathbf{\,e}}_{p^2}}
\def\em{{\mathbf{\,e}}_m}

\def\Res{\mathrm{Res}}
\def\Orb{\mathrm{Orb}}

\def\vec#1{\mathbf{#1}}
\def \va{\vec{a}}
\def \vb{\vec{b}}
\def \vn{\vec{n}}
\def \vu{\vec{u}}
\def \vv{\vec{v}}
\def \vw{\vec{w}}
\def \vz{\vec{z}}
\def\flp#1{{\left\langle#1\right\rangle}_p}
\def\sM {\mathsf {M}}

\def\md{{\sf{m.d.}}}

\def\sfG {\mathsf {G}}
\def\sfK {\mathsf {K}}

\def\rad{\mathrm{rad}}
 \newcommand{\Fp}{\mathbb F_p}
 \newcommand{\Gal}{\operatorname{Gal}}

\def\mand{\qquad\mbox{and}\qquad}

%\title[Additive energy of cyclic matrix groups]
%{Additive energy of cyclic matrix groups and character
% sums with matrix exponential functions}

%\title[Recurrence sequences with zeros]{Frequency of linear recurrence sequences with zeros}

\title[Solvable  $\cS$-unit equations and recurrences   with zeros]{Counting solvable $\cS$-unit equations and linear recurrence sequences with zeros}

\author[A. Ostafe] {Alina Ostafe}
\address{AO:  School of Mathematics and Statistics, University of New South Wales, Sydney, NSW 2052, Australia}
\email{alina.ostafe@unsw.edu.au}

\author[C. Pomerance]{Carl Pomerance}
\address{CP: Department of Mathematics, Dartmouth College, 
Hanover, NH 03755-3551, USA} 
\email{carlp@gauss.dartmouth.edu} 

\author[I. E. Shparlinski] {Igor E. Shparlinski}
\address{IS: School of Mathematics and Statistics, University of New South Wales, Sydney, NSW 2052, Australia}
\email{igor.shparlinski@unsw.edu.au}

\begin{abstract} 
{We show that only a rather small proportion of linear equations are solvable in elements 
of a fixed finitely generated subgroup of a multiplicative group of a number field. The argument is based on 
modular techniques combined with a classical idea of P.~Erd{\H o}s~(1935). We then use similar ideas
to get a tight upper bound on the number of linear recurrence sequences which attain a zero value. }
\end{abstract}



\maketitle
%
%%\tableofcontents
%

\section{Motivation and set-up}

Recently, there has been several works counting soluble (globally or locally) polynomial Diophantine 
equations in various families, see~\cite{AkBha,Brow,BrDi,BLePS,DiMa,KPSS, LLY,Lou} and references therein.

Here we address a similar question for families of linear equations in elements of finitely generated groups, 
which are also known as {\it $\cS$-unit equations\/}, we refer to~\cite{EvGy} for background.  

Namely, let  $\Gamma \subseteq \K^*$ be a finitely generated multiplicative subgroup of $\K^*$, where $\K$ is a number field 
of degree $d = [\K:\Q]$ over $\Q$.

We also fix an integral basis $\omega_1, \ldots, \omega_d$ of the ring of integers $\Z_\K$ of $\K$, and for 
an integer $H \ge 0$
we consider the set 
\[
\cA(H) = \{\alpha= u_1\omega_1+ \ldots + u_d \omega_d:~
u_i \in [-H,H] \cap\Z,\ i =1, \ldots, d\}. 
\]
Clearly, $\cA(H)$ is of cardinality $\#\cA(H) = \(2H+1\)^d$.

Finally, we denote by $Z_k(\Gamma, H)$ the number of $k$-tuples of coefficients $\(\alpha_1, \ldots, \alpha_k\) \in \cA(H)^k$, 
 such that the equation 
\begin{equation}
\label{eq:S-unit Eq}
\alpha_1\vartheta_1 + \ldots + \alpha_k \vartheta_k = 0, \qquad \vartheta_1,  \ldots, \vartheta_k \in \Gamma,
\end{equation}
has a solution.  
Our first main result, Theorem \ref{thm:Zeros}, estimates $Z_k(\Gamma,H)$ with
a power savings.

We note that the question of estimating  $Z_k(\Gamma, H)$ is  somewhat 
dual to the scenario of~\cite{ShSt} where, for $\K = \Q$ and $k=3$, the coefficients 
are fixed but $\Gamma$ varies among  groups generated by $r$ primes in a given interval. 

We also use similar ideas to bound  the number of  linear recurrence sequences, 
which have a zero  in their value set.

Let 
\begin{equation}
\label{eq:f}
f(X) = X^k -  c_{k-1} X^{k-1} -\ldots -  c_{0} \in \Z[X], \quad c_0\ne0,
\end{equation}
and let $\cL_f$ denote the set  of all  linear recurrence sequences
\[
\vu = \(u(j)\)_{j=1}^\infty
\] 
with the characteristic polynomial  $f$, that is, with
\begin{equation}
\label{eq:LRS}
u(j+k) = c_{k-1} u(j+k-1) +\ldots +  c_{0} u(j) , \qquad j =1,2, \ldots,
\end{equation}
and integer  initial values $u(1), \ldots, u(k)$ not all zero. 

If there are no roots of unity among   the ratios
of distinct roots of its characteristic polynomial $f$, then all   sequences $\vu\in \cL_f$ are called 
{\it non-degenerate\/}. 

By the classical  Skolem--Mahler--Lech theorem, any  non-degenerate  linear recurrence sequence contains only
finitely many  zeros (see~\cite{AmVia} for the strongest known bound). Hence, there is an integer $n_0>0$, depending only on $\vu$, such that $u(n) \ne 0$  for all $n\ge n_0$. 

It is also easy to see that ``typical'' polynomials $f$ correspond to non-degenerate 
linear recurrence sequences, thus having a zero is a rare  event. Our second main result, Theorem~\ref{thm:Zeros-LRS}, 
implies that in fact 
typically   linear recurrence sequences $\vu  \in \cL_f$ (whether degenerate or not) 
do not have zeros at all.

For $U \ge 1$,  we give an upper bound on   the number $Z_f(U)$ of   linear recurrence sequences
$\vu  \in \cL_f$ with   integer  initial values $\(u(1), \ldots, u(k)\) \in [-U,U]^k$ for which $u(n) = 0$ 
for some $n$. 

Our approach to bounding $Z_k(\Gamma, H)$ and   $Z_f(U)$ is based on a modular technique
and also on generating a reasonably dense sequence of  integers with small values of the Carmichael $\lambda$-function
and composed from arbitrary sets of primes of positive relative density, see Lemma~\ref{lem:Small-lambda}
below.  (The Carmichael $\lambda$-function at a positive integer $n$ returns the exponent
of the group $(\Z/n\Z)^*$.) 

The argument we use dates back to work of Erd{\H o}s~\cite{Erd}; it has also been used in 
various modifications in a number of other works, see, for example,~\cite{EPS}.  

We also note that in the case of $Z_f(U)$, surprisingly enough, the modular approach gives 
an essentially tight bound. 

\section{Main results}

We first give an upper bound on $Z_k(\Gamma, H)$ with a power savings.

We always assume that  $d$, $k$,  the  subgroup $\Gamma$
and the characteristic polynomial $f \in \Z[X]$ are fixed.  In particular all implied
constants and the functions denoted by the $o$-symbol may depend on them.

\begin{thm}
\label{thm:Zeros} Let $\K$ be a number field 
of degree $d = [\K:\Q]$ over $\Q$ and let 
$\Gamma \subseteq \K^*$  be a  finitely generated group.   Then, as
$H\to\infty$,
\[
Z_k(\Gamma, H) \le H^{dk- 1+ o(1)}.
\]
\end{thm}

\begin{rem} 
\label{rem:S-unit Eq} 
{\rm Examining the proof of Theorem~\ref{thm:Zeros} one can notice that 
similar ideas can allow us to investigate equations with coefficients which are arbitrary 
algebraic numbers of the form $\alpha/\beta$ with $\alpha, \beta \in \cA(H)$, or 
of the form  $\alpha/b$ with $\alpha  \in \cA(H)$ and $b \in \{1, \ldots, H\}$. }
\end{rem}

A variation of the argument used in the proof of Theorem~\ref{thm:Zeros} 
also gives the following tight bounds.

\begin{thm}
\label{thm:Zeros-LRS} Let $f\in\Z[X]$ be defined by~\eqref{eq:f}.  
If $f$ is separable, then, as $U\to\infty$, 
\[
U^{k -1} \le Z_f(U) \le U^{k -1 + o(1)}.  
\]
\end{thm}

\begin{rem} 
\label{rem:un=a}
{\rm It is easy to see that our argument also applies to
inhomogeneous versions of the equations~\eqref{eq:S-unit Eq} with some fixed  $\rho \in \K$ 
on the right hand side 
and  to counting linear recurrence  sequences which contain a 
prescribed value $b\in \Z$ and leads to the same upper bounds 
(uniformly in  $\rho$ and $b$). }
\end{rem}

\section{Small values of the Carmichael $\lambda$-function}

We recall that for an integer $n\ge 2$, the Carmichael $\lambda$-function $\lambda(n)$
is the smallest positive integer $m$ such that $a^m\equiv 1\pmod n$ for all $a$ coprime to $n$.

We say that a set of primes  $\cP$ is of relative density $\delta$ if
\[
\#\(\cP\cap [1,x]\) \sim\delta\pi(x), \qquad \text{as}\ x\to\infty, 
\]
where, as usual, $\pi(x)$ is the number of primes up to $x$.  
Let $x$ be large, and let 
\[
y=\log x/\log\log x,\quad M=\lcm[1,2,\dots,\lfloor y\rfloor],
\]
so that $M=x^{(1+o(1))/\log\log x}$ as $x\to\infty$.  
Recall that if $n=p_1\dots p_k$ where
$p_1,\dots,p_k$ are distinct primes, then 
\[
\lambda(n)=\lcm[p_1-1,\dots,p_k-1].
\]
Thus, if each $p_i-1\mid M$, then $\lambda(n)\mid M$ and $\lambda(n)\le x^{(1+o(1))/\log\log x}$.

Below, we also allow all constants and $o$-functions to depend 
on  the real positive parameter $\varepsilon$ and the set of primes $\mathcal P$.

\begin{lem}
\label{lem:Small-lambda}
Let $\varepsilon>0$ be arbitrarily small and
suppose $\cP$ is a set of primes of relative density $\delta>0$.  There is a number $x_0$
(depending on $\varepsilon$ and 
$\cP$) 
such that if $x>x_0$, there is a squarefree integer
$n\in\((1-\varepsilon)x,x\right]$ composed solely of primes $p$ from $\cP$ and such that $p-1\mid M$.
In particular, $\lambda(n)\le x^{(1+o(1))/\log\log x}$.
\end{lem}

\begin{proof}
Let $\cQ=\{p\in\cP:~p-1\mid M\}$. 
First note that $\cP$ and $\cQ$ agree up to $y$.  Thus, if
$x_0$ is large enough (depending on $\varepsilon$ and $\cP$)
and $\log\log x < t < y$, then the number of elements  $p\in \cQ$ such that
 $p\le t$ is in the interval $\((1-\varepsilon)\delta t/\log t,(1+\varepsilon)\delta t/\log t\)$.  We first show that
this continues for $t$ up to 
\[
z=\log x\log\log x.
\]
  Indeed, if $p\in\cP\setminus\cQ$, then
$p-1$ is divisible either by a prime $q>y$ or by a prime power $\ell^j>y$, for a prime $\ell$  and integer $j\ge2$. 
The number of primes $p\le t$ satisfying  the second condition is at most 
\[\sum_{\substack{\ell^j > y\\\ell~\text{prime}\\j\ge 2}}t/\ell^j 
\le \sum_{\substack{m^j > y\\ m \in \N\\ j \ge 2}}t/m^j \ll t/y^{1/2}={o\(\pi(t)\)} \] for $t \le z$.

The same is true for the first condition
as we now show.  If $q\mid p-1$, write $p-1=aq$, so if $p\le t$ and $q>y$, then 
$a<t/y$.  Assume that $y< t\le z$, fix an integer $a<t/y$, and count primes $q\le t/a$ with $aq+1$ prime.
By Brun's sieve, the number of such primes $q$ is
$O\((t/\varphi(a))(\log(t/a))^{-2}\)$, where $\varphi(a)$ is the Euler function, see, for example,~\cite[Proposition~6.22]{FrIw}
for a much more general and precise statement.  Since $y< t\le z$,
we have $a\le(\log\log x)^2$ and $\log(t/a)\sim\log t\sim\log\log x$.  Since 
\[\sum_{a<t/y}1/\varphi(a)
\ll \log\log\log x\sim\log\log t,
\] 
we have
\[
\# \{p\in\cP\setminus\cQ:~p \le t\} \ll \pi(t)\log\log t/\log t=o(\pi(t)).
\]

Let $n_1$ be the product of all of the primes in $\cQ\cap[1,z]$, so that
$\lambda(n_1)\mid M$ and  
{$n_1\ge x^{(1-c_0\varepsilon)\delta\log\log x}$, for some absolute constant $c_0$. 
Thus, assuming that $\varepsilon$ is small enough,} we see that $n_1$
is quite a bit larger than $x$.  Remove the top primes from $n_1$ stopping just before removing the
next one would drop the number below $x(\log x)^{1/2}$, and denote this number by $n_2$.
Thus, $x(\log x)^{1/2}<n_2<x(\log x)^{1/2}z$.   Let
$g=n_2/x$ so that $(\log x)^{1/2}<g<(\log x)^{1/2}z$.

Since {$\cP$} has a positive relative density in the primes, there are members $p_1,p_2$
in $\cP$ with $p_1\sim p_2\sim g^{1/2}$, and in particular, we can take
$p_1,p_2\in \((1-\varepsilon/2)g^{1/2},g^{1/2}\right ]$.  Also, since $g^{1/2}<y$, we have
$p_1,p_2\in\cQ$.
 Since 
\[(\log x)^{1/4}<g^{1/2}<(\log x)^{1/4}z^{1/2}<y,
\] we have
$p_1p_2\mid n_2$.  Let $n=n_2/p_1p_2$.  Then $n\in\((1-\varepsilon)x,x\right]$, which completes the proof.
\end{proof}

\section{Proof of Theorem~\ref{thm:Zeros}} 
We fix the basis elements $\omega_1, \ldots, \omega_d$ of   $\Z_\K= \Z[\omega_1, \ldots, \omega_d]$ 
and let $r$ be the rank of  $\Gamma$.

We first observe that if the prime $p$ splits completely in $\K$ then the residue ring 
$\Z_\K/\fP$ modulo a prime ideal $\fP$ of $\Z_\K$ lying over $p$ is isomorphic to the finite field $\F_p$ of $p$ elements. 
This means that for any $\alpha \in \Z_\K$, there is an integer $a_\fP \in \Z$ with 
\[
\alpha \equiv a_\fP \pmod \fP.
\]

Let $\cP$ be the set of primes  which split completely in $\K$ and also are relatively prime (as ideals in $\Z_\K$) to the basis elements $\omega_1, \ldots, \omega_d$ of   $\Z_\K$ 
and to the prime ideals appearing in the factorisation of the generators $\gamma_1, \ldots, \gamma_{r}$ of $\Gamma$, seen as fractional ideals in $\K$.

Therefore, for each $p\in\cP$ and prime ideal $\fP$ of $\Z_\K$ lying over $p$ there are integers $w_{i,\fP} \in \Z$, $i=1,\ldots,d$, with 
\begin{equation}
\label{eq:omega P}
\omega_i \equiv w_{i,\fP} \pmod \fP, \quad i=1,\ldots,d,
\end{equation}
and the equation~\eqref{eq:S-unit Eq} implies that  
\begin{equation}
\label{eq:lin cong P}
a_{1,\fP} \prod_{j=1}^{r} g_{j,\fP}^{s_{1j}}  + \ldots + a_{k,\fP}  \prod_{j=1}^{r} g_{j,\fP}^{s_{k,j}}  \equiv 0 \pmod \fP, 
\end{equation}
with some integers $s_{ij}$, $i=1, \ldots, k$, $j=1, \ldots, r$, and some integers $a_{i,\fP}\equiv \alpha_i\pmod \fP$, $i=1, \ldots, k$, and integers $g_{j,\fP}\equiv \gamma_j \pmod \fP$, $j=1, \ldots, r$. 

Since the left hand side of~\eqref{eq:lin cong P} is an integer, this also implies that
\begin{equation}
\label{eq:lin cong p}
a_{1,\fP} \prod_{j=1}^{r} g_{j,\fP}^{s_{1j}}  + \ldots + a_{k,\fP}  \prod_{j=1}^{r} g_{j,\fP}^{s_{kj}}  \equiv 0 \pmod p.
\end{equation}

 Since a prime $p$ splits completely in $\K$ if and only if it splits completely in the Galois closure of $\K$, 
 see~\cite[Corollary, Page~108]{Marcus}, by  the Chebotarev Density Theorem applied to the Galois closure of $\K$,  see~\cite[Theorem~21.2]{IwKow},  the set $\cP$ is of positive relative density.
 
 We choose now $n$  as in Lemma~\ref{lem:Small-lambda} applied with $x = H$, and since the 
 congruence~\eqref{eq:lin cong p} holds for each $p\in \cP$, by the Chinese Reminder Theorem we obtain
\begin{equation}
\label{eq:lin cong n}
a_1 \prod_{j=1}^{r} g_j^{s_{1j}}  + \ldots + a_k  \prod_{j=1}^{r} g_j^{s_{kj}}  \equiv 0 \pmod n,
\end{equation}
for some integers $a_i$, $i=1, \ldots, k$, and $g_j$, $j=1, \ldots, r$, such that
\[
a_i\equiv a_{i,\fP}\pmod \fP \mand g_j\equiv g_{j,\fP}\pmod \fP
\]
for any prime ideal $\fP$ of $\Z_\K$ lying over a prime $p\mid n$.
 
Hence the integer vector $(a_1, \ldots, a_k)$ satisfies at least one of  at most $\lambda(n)^{kr}$ 
possible nontrivial linear congruences~\eqref{eq:lin cong n}, and thus takes at most 
$\lambda(n)^{kr} n^{k-1}$
possible values modulo $n$. 

For a given $(a_1,\dots,a_k)$ as above we are left to count the number of possibilities
$(\alpha_1,\dots,\alpha_k)\in\cA(H)^k$ such that
\[
\alpha_i\equiv a_{i,\fP}\pmod \fP
\]
for all prime ideals $\fP$ of $\Z_\K$ dividing $n$.

Let $\alpha\in\cA(H)$, that is, $\alpha= u_1\omega_1+ \ldots + u_d \omega_d$, $u_i \in \Z\cap [-H,H]$, $ i =1, \ldots, d$. Let $\fP$ be a prime ideal  of $\Z_\K$  lying over a prime $p\in\cP$  
 and let $a_\fP \in \Z$ satisfy 
 \begin{equation}
\label{eq:alpha=a mod P}
  \alpha \equiv a_\fP\pmod \fP. 
\end{equation}
 From~\eqref{eq:alpha=a mod P} and recalling the notation~\eqref{eq:omega P}, we obtain
\[
u_1w_{1,\fP}+ \ldots + u_d w_{d,\fP} \equiv a_\fP \pmod \fP.
\]
Hence, as above, this congruence holds modulo $p$ and thus modulo $n$ chosen above, that is, we have
\[
u_1w_{1}+ \ldots + u_d w_{d} \equiv a_\fP \pmod n,
\]
such that $w_i\equiv w_{i,\fP} \pmod \fP$, $i=1,\ldots,d$, and 
where by our definition of $\cP$ 
 we have $\gcd(w_{1}\cdots w_{d}, n) =1$. 

We now see that for $n \le H$ there are $O(H^d/n)$ elements $\alpha\in \cA(H)$, which 
satisfy~\eqref{eq:alpha=a mod P}. 

Therefore, recalling that there are at most $\lambda(n)^{kr} n^{k-1}$ possibilities for $(a_1,\ldots,a_k)$, 
we obtain 
\[
Z_k(\Gamma, H) = O\( \lambda(n)^{kr} n^{k-1}  \(H^d/n\)^k \).  
\]  
Since $\lambda(n) = n^{o(1)} = H^{o(1)}$, and by Lemma~\ref{lem:Small-lambda}, we have $n>(1-\varepsilon)H$, for $\varepsilon>0$ arbitrarily small, we conclude the proof.

\section{Proof of Theorem~\ref{thm:Zeros-LRS}}

The lower bound is obvious from considering initial values with, for example, $u(1) = 0$.

To establish the upper bound, we  choose  the set $\cP$ of all  primes 
$p$,  
 such that $f(X)$ splits completely modulo each $p\in \cP$. By the Chebotarev Density Theorem~\cite[Theorem~21.2]{IwKow} applied to the splitting field of $f$, {the set} $\cP$ is of relative density $\delta \ge 1/k!$.

By removing at most finitely many members of $\cP$ we may  assume that  any $p \in \cP$ is 
relatively prime to $f(0)$ and the discriminant of $f$. 
This means that any linear recurrence sequence
$\vu = \(u(j)\)_{j=1}^\infty$ with the characteristic polynomial  $f$, taken modulo $p$, is a simple linear recurrence and thus can be written 
as 
\begin{equation}
\label{eq:LRS Mod p}
u(j) \equiv \sum_{\nu=1}^k a_{\nu,p} g_{\nu,p}^j \pmod p, \qquad j =1, 2, \ldots, 
\end{equation}
for some integers $a_{\nu,p}$ and
 distinct  modulo $p$ integers $g_{\nu,p}$  such that  $\gcd(g_{\nu,p},p) = 1$, 
see~\cite[Chapter~3]{EvdPSW} for more details.   

We now take $n$ as in Lemma~\ref{lem:Small-lambda} applied with $x = U$.

By the Chinese Remainder Theorem, we derive from~\eqref{eq:LRS Mod p} that 
\[
u(j) \equiv \sum_{\nu=1}^k A_{\nu} G_{\nu}^j \pmod  n, \qquad j =1, 2, \ldots,
\]
for some integers $A_{\nu}$ and 
 distinct modulo $n$ integers $G_{\nu}$  such that $\gcd(G_{\nu},n) = 1$. Therefore,  $u(j)$, $j=1,2,\ldots$, is purely periodic modulo $n$ with period 
\begin{equation}
\label{eq:Per mod n}
t \le  \lambda(n).
\end{equation}

To represent $\vu$ using the initial values $u(1),\ldots,u(k)$, 
we define the sequences $\vw_i  \in \cL_f$, $i =1, \ldots, k$, with initial values 
\[
w_i(j) = \begin{cases} 1, & \text{if } i=j,\\
 0 , & \text{if } i\ne j,
 \end{cases}, \qquad j =1, \ldots, k.
 \] 
 It is now obvious that for any $\vu  \in \cL_f$  we have
\begin{equation}
\label{eq: Span}
u(j) = \sum_{i=1}^k u(i)  w_i(j), \qquad j =1,2, \ldots.
\end{equation}
Indeed, both the left and the right hand-sides of the equation~\eqref{eq: Span} belong to 
$\cL_f$  and have the same initial values; hence they coincide for all $j$. 

In particular~\eqref{eq: Span} implies that for any integer  $m\ge 1$ we have 
\begin{equation}
\label{eq: GCD}
\gcd \(w_1(m), \ldots,   w_k(m), p\) =1
\end{equation}
for $p \in \cP$. Indeed, writing~\eqref{eq: Span} for shifts of say $\vw_1$, that is, writing 
\[
w_1(j+h) = \sum_{i=1}^k w_1(i+h)  w_i(j), \qquad h = 0, \ldots, k-1,
\]
we see that if~\eqref{eq: GCD} fails then for some $m$ we have
\[
p\mid w_1(m+h), \qquad h = 0,\ldots,  k-1.
\] 
Next,  the recurrence relation~\eqref{eq:LRS} implies that   
$p\mid w_1(m+k)$, and similarly $p\mid w_1(j)$ for all $j \ge m$. Recalling that  $\vw_1$ is periodic, we conclude that $p\mid w_1(1)$, which is a contradiction.

If  $\vu  \in \cL_f$  has a zero, then, by periodicity,  for some positive integer $j \le t$, 
 the representation~\eqref{eq: Span} implies
\[
 \sum_{i=1}^k u(i)  w_i(j) \equiv 0 \pmod n.
 \]
Recalling~\eqref{eq: GCD}, we see that, since by our construction $n\le U$,  this is possible for at 
most $O(U^k/n)$ initial values 
$\(u(1), \ldots, u(k)\) \in [-U,U]^k$. Hence, by~\eqref{eq:Per mod n}, 
\[
Z_f(U) = O\( t U^k/n\) = O\(\lambda(n) U^k/n\)   
\]
and since, as before, by Lemma~\ref{lem:Small-lambda}, we have $n>(1-\varepsilon)U$, for $\varepsilon>0$ arbitrarily small, we conclude the proof. 

\section*{Acknowledgments}

The work of A.O. and I.S. was supported, in part, by the Australian Research Council Grant DP230100530. A.O. also
gratefully acknowledges the hospitality and support of the Institut des
Hautes \'Etudes Scientifiques, where  her work has been carried out.

\end{document}